\documentclass{amsart}

\usepackage{tikz-cd}

\usepackage{amsmath}
\usepackage{amssymb}
\usepackage{amsthm}
\newtheorem{thm}{Theorem}[section]
\newtheorem{cor}[thm]{Corollary}

\newtheorem{prop}[thm]{Proposition}

\newtheorem{qu}[thm]{Question}
\newtheorem{ex}[thm]{Example}

         \def\a{\alpha} \def\b{\beta} \def\d{\delta}
\def\g{\gamma}  
  
 \def\sm{\setminus} \def\o{\omega}
\def\to{\rightarrow}  \def\imp{\Rightarrow}
\def\r{\upharpoonright}  \def\cl{\overline}
  \def\0{\emptyset}
 
\def\<{\langle} \def\>{\rangle}

\def\F{\mathcal F}
\def\V{\mathcal V}
\def\vf{\varphi}

\begin{document}

\title{Continuous Selections of Lower semicontinuous Set-valued Mappings}

\author[Z. Feng]{Ziqin Feng}
\address{Department of Mathematics and Statistics\\Auburn University\\Auburn, AL~36849, USA}
\email{zzf0006@auburn.edu}

\author[G. Gruenhage]{Gary Gruenhage}
\address{Department of Mathematics and Statistics\\Auburn University\\Auburn, AL~36849, USA}
\email{gruengf@auburn.edu}

\author[R. Shen]{Rongxin Shen}
\address{Department of Mathematics, Taizhou University, Taizhou 225300, P. R. China}
\email{srx20212021@163.com}

\dedicatory{Dedicated to A.V. Arhangel'skii on the occasion of his 80th birthday.}

\begin{abstract} A space $X$ is strongly $Y$-selective (resp., $Y$-selective)  if every lower semicontinuous mapping from $Y$ to the nonempty subsets (resp., nonempty closed subsets) of $X$ has a continuous selection.  We also call $X$ (strongly) $C$-selective if it is (strongly) $Y$-selective for any countable space $Y$, and (strongly) $L$-selective if it is (strongly) ($\omega+1$)-selective.   E. Michael showed that every first countable space is strongly $C$-selective.  We extend this by showing that every $W$-space in the sense of the second author is strongly $C$-selective.  We also show that every GO-space is $C$-selective, and that every $L$-selective space has Arhangel'skii's  property $\alpha_1$.  We obtain an example under $\mathfrak p=\mathfrak c$ of a strongly $L$-selective space that is not $C$-selective, and we show that it is consistent with and independent of ZFC that a space is strongly $L$-selective iff it is $L$-selective and Fr\'echet.  Finally, we answer a question of the third author and Junnila by showing that the ordinal space $\omega_1 +1$ is not self-selective. \end{abstract}

\date{\today}

\maketitle

\section{introduction}
 For topological spaces $Y$ and $X$, a set-valued mapping $\varphi:Y\to (\mathcal{P}(X)\setminus\{\emptyset\})$ is said to be {\it lower semicontinuous} ({\it l.s.c.}) if for every open $U\subset X$, the set $\varphi^{-1}(U)=\{y\in Y: \varphi(y)\cap U\neq\emptyset\}$ is open in $Y$. Equivalently, $\varphi(\cl{A})\subseteq \cl{\varphi(A)}$ for every $A\subset Y$, where by $\varphi(B)$ we mean $\bigcup\{\varphi(y):y\in B\}$. We say $\varphi$ is {\it closed-valued} if $\varphi(y)\in \mathcal{F}(X)$ for each $y\in Y$, where $\mathcal{F}(X)$ is the family of all nonempty closed  subsets of $X$. A {\it selection} of  $\varphi$  is a mapping $f:Y\to X$ such that $f(y)\in \varphi(y)$ for every $y\in Y$. Michael \cite{m81} proved the following result.

 \begin{thm}\label{m81} Assume that $Y$ is countable and regular, $X$ is first countable, and $\varphi: Y\rightarrow (\mathcal{P}(X)\setminus\{\emptyset\})$ is l.s.c.. Then $\varphi$ has a continuous selection. \end{thm}


 Some applications and developments of Michael's theory of continuous selection are discussed in \cite{RS94}. One of the problems in selection theory is what conditions should be imposed on $X$ and $Y$ such that any l.s.c. closed-valued mapping from $Y$ to $X$ has a continuous selection. In this note, we say that the space $X$ is  {\it (strongly) $Y$-selective} if every l.s.c. mapping from $Y$ to  $\mathcal{F}(X)$ (respectively, $\mathcal{P}(X)\setminus\{\emptyset\}$ ) has a continuous selection.  We also say that $X$ is {\it $L$-selective}  if it is ($\omega+1$)-selective . If $X$ is (strongly) $Y$-selective for all countable regular $Y$, then we say that it is {\it (strongly) $C$-selective}. By Michael's theorem above, any first countable space is strongly $C$-selective. In this paper, we discuss conditions that make $X$ (strongly) $L$-selective or (strongly) $C$-selective, and illustrative examples are constructed.

A space $X$ is said to be {\it self-selective} if it is $X$-selective.  In Section~\ref{l-sel}, we prove that $\omega_1+1$ is not self-selective; this answers negatively a question raised by Shen and Junnila in \cite{sj}. We also show that every $L$-selective space has Arhangel'skii's property $\alpha_1$.

In Section~\ref{c-sel}, we extend Michael's theorem by showing that any $W$-space (see the definition in Section~\ref{c-sel}) is strongly $C$-selective. We also show that every GO-space is $C$-selective, and an example of a strongly $C$-selective non-$W$-space is constructed.

In Section~\ref{l-sel-nq-sel}, an example of $L$-selective but not $\mathbb Q$-selective space is constructed assuming $\mathfrak p=\mathfrak c$, where $\mathbb Q$ is the space of rational numbers.
In Section~\ref{L-not-strong-L},  we prove that it is consistent with and independent of ZFC that a space is strongly $L$-selective iff it is $L$-selective and Fr\'echet. In the final section, we give a chart summarizing our main results, along with a list of open questions.

All spaces are assumed to be at least $T_1$.

\section{Basic properties}\label{basic}

We start with some general properties of $Y$-selective spaces. The following result follows directly from the definition.

\begin{prop}\label{c_subs_r} If a space $X$ is strongly $Y$-selective (resp., $Y$-selective), then so is any subspace (resp., closed subspace) of $X$.  \end{prop}

\begin{prop}\label{c_subs_d} Let $X$, $Y$, and $Z$ be spaces, with $Z\subset Y$,  and suppose $Z$ is a closed subspace of $Y$.  If $X$ is (strongly) $Y$-selective, then it is also (strongly) $Z$-selective. \end{prop}

\begin{proof} Let $\phi$ be an l.s.c. mapping from $Z$ to the nonempty closed subsets (resp., the nonempty subsets) of $X$ . Then define the mapping $\phi^\ast$ by $\phi^\ast(y)=\phi(y)$ for $y\in Z$ and $\phi^\ast(y)=X$ for $y\in Y\setminus Z$. Take a subset $A$ of $Y$. If $A\setminus Z\neq \emptyset$, it is clear that $\phi^{\ast}(\overline{A})\subset \overline{\phi^\ast(A)}=X$. If $A\subset Z$, then  $\phi^{\ast}(\overline{A})\subset \overline{\phi^\ast(A)}$ since $\phi$ is l.s.c. Therefore, $\phi^\ast$ is l.s.c. The restriction to $Z$ of a continuous selection of $\phi^\ast$  is a continuous selection of $\phi$. Hence $X$ is $Z$-selective. \end{proof}

By Proposition~\ref{c_subs_r}, strongly $L$-selective and strongly $C$-selective spaces are hereditary classes, and without ``strongly" they are closed hereditary.  We now consider open subspaces.

Let $K$ be a compact scattered space. Recall the Cantor--Bendixson process: let $K^{(0)}$  be the set of the isolated elements in $K$, and
inductively let
$K^{(\beta)}$ be the set of isolated elements in $K\setminus
(\bigcup_{\gamma<\beta}K^{(\gamma)})$. Also write $K^{(\leq
\beta)}$ for $\bigcup_{\gamma\leq \beta} K^{(\gamma)}$ and $K^{(\geq
\beta)}$ for $\bigcup_{\gamma\geq \beta} K^{(\gamma)}$. This process terminates for some minimal ordinal $\alpha+1$, called the {\it scattered height} of $K$, when   $K^{(\alpha+1)}=\emptyset$. It is known that any compact scattered space is zero-dimensional (has a base of clopen sets).

\begin{prop} Let $K$ be a compact scattered space. If a space $X$ is regular and $K$-selective, then any open subspace of $X$ is $K$-selective.  \end{prop}

\begin{proof}The result clearly holds if the scattered height of $K$ is $1$ in which case $K$ is a space of finite isolated points. Suppose the result holds for any compact space with scattered height $<\alpha+1$. Let $X'$ be an open subspace of $X$. Let $\varphi$ be an l.s.c. closed-valued mapping from $K$ to $X'$. Since $K^{(\alpha)}$ is finite, there is an open subset $U$ of $X$ such that $\varphi(x)\cap U\neq \emptyset$ for each $x\in K^{(\alpha)}$ and $\overline{U}\subset X'$.  Since $\varphi$ is l.s.c., $\varphi^{-1}(U)$ is open in $K$. Again, since $K^{(\alpha)}$ is finite, there is a clopen subset $V$ of $K$ satisfying $K^{(\alpha)}\subset V$. Then $X$ is $V$-selective by Proposition~\ref{c_subs_d} and hence $\overline{U}$ is $V$-selective by Proposition~\ref{c_subs_r}. Define $\varphi_1$ by $\varphi_1(x)=\varphi(x)\cap \overline{U}$ for each $x\in V$. Then $\varphi_1$ is an l.s.c. mapping from $V$ to $\mathcal{F}(\overline{U})$.  Then there is a continuous selection $f_1$ of $\varphi_1$. Also, $K\setminus V$ is a compact scattered space with scattered height clearly $<\alpha+1$. Define $\varphi_2$ by $\varphi_2(x)=\varphi(x)$ for each $x\in K\setminus V$; then $\varphi_2(x)$ is l.s.c. So by the inductive  assumption, there is a continuous selection $f_2$ of $\varphi_2$.

Let $f=f_1\cup f_2$. We claim that $f$ is a continuous selection of $\varphi$. It is clear that $f(x)\in\varphi(x)$ for each $x\in K$.  Take a subset $A$ of $K$. Let $A_1=A\cap V$ and $A_2=A\cap (K\setminus V)$. Then since $V$ is clopen, $f(\overline{A})=f(\overline{A_1})\cup f(\overline{A_2})\subset \overline{f(A_1)}\cup \overline{f(A_2)}=\overline{f(A_1)\cup f(A_2)}=\overline{f(A)}$. So $f$ is continuous.   \end{proof}

\begin{cor} Every open subspace of an $L$-selective space is $L$-selective. \end{cor}

\begin{prop} Every open subspace of a $C$-selective space is $C$-selective.  \end{prop}

\begin{proof} Suppose $X$ is a $C$-selective space and $Z$ is an open subspace of $X$.  Let $\varphi:Y\to \F(Z)$ be l.s.c., where $Y$ is a countable regular space. We will prove that $\varphi$ has a continuous selection.

To this end, for each $y\in Y$ pick a point $x_y\in \varphi(y)$, and define $\varphi_y:Y\to \F(X)$ by $\vf_y(y')=\cl{\vf(y')}$ if $y'\neq y$ and $\vf_y(y)=x_y$.  Suppose $U$ is an open subset of $X$.  Since $\vf(y')\cap U\neq \0 \iff \cl{\vf(y')}\cap U\neq 0$, it follows that $\vf_y^{-1}(U)$ is either $\vf^{-1}(U)$ or $\vf^{-1}(U)\sm\{y\}$ and thus is open.  Hence $\vf_y$ is l.s.c., and so has a continuous selection $f_y$.  Since $f_y(y)=x_y\in Z$, $O_y=f_y^{-1}(Z)$ is an open neighborhood of $y$.  Since $Y$ is countable and regular, it is zero-dimensional and Lindel\"of hence ultraparacompact.  Let $\V$ be a pairwise-disjoint clopen refinement of $\{O_y:y\in Y\}$.  For each $V\in \V$, let $y(V)\in Y$ be such that $V\subset O_{y(V)}$.  Then the function $f:Y\to X$ defined by $f\r V=f_{y(V)}\r V$ for each $V\in \mathcal V$ is a continuous selection for $\vf$.
\end{proof}

The following is a local property of $L$-selective spaces (and hence $C$-selective spaces too). It follows from this result that a (countably) compact space with no convergent sequences, for example $\beta\omega$, cannot be $L$-selective.

\begin{prop} Let $X$ be strongly $L$-selective (resp., $L$-selective) and $p\in X$.  Suppose  $\{H_n:n\in\omega\}$ is a countable collection of subsets (resp., closed subsets) with $p\in \cl{\bigcup_{n\in\omega}H_n}$.
Then there is a sequence of points  in $\bigcup_{n\in\omega}H_n$ which converges to $p$.  \end{prop}
\begin{proof} Suppose $X$, $p$, and $\{H_n:n\in\omega\}$ satisfy the hypotheses. Define a mapping $\psi$ from $\omega+1$ to the nonempty (closed) subsets of $X$ by $\psi(n)=\bigcup_{i\leq n}H_i$ and $\psi(\omega)=\{p\}$.  It is easy to check that $\psi$ is l.s.c., so let $f$ be a continuous selection.  Then $f(n)$, $n\in\omega$, converges to $f(\omega)=p$.
\end{proof}

Some easy consequences of this proposition are listed below.

\begin{cor}\begin{enumerate}
    \item[(a)]\label{count_tight} Every countably tight $L$-selective space is Fr\'echet;
    \item[(b)]\label{count_subspace}Every countable subspace of an $L$-selective space is Fr\'echet;
    \item[(c)]\label{sl_f} Any strongly $L$-selective space is  Fr\'echet.
\end{enumerate} \end{cor}





\section{on self-selective spaces}\label{l-sel}
 In this section, we first give a negative answer to the question asked in \cite{sj} whether every compact ordinal space is self-selective.

\begin{prop}$\omega_1 +1$ is not self-selective.
\end{prop}
\begin{proof}Let $C$ be the set of all ordinals in $\omega_1$ that are limits of limit ordinals.   (For example, $\omega$ is not in $C$ but $\omega^2$ is in $C$. )  Note that $C$ is club (closed and unbounded).

For each $\delta$ in $C$, let $\delta(n)$, $n=1,2,\dots$, be an increasing sequence of ordinals converging to $\delta$.  Then define $\varphi$ as follows:
\begin{enumerate}
  \item[(i)] $\varphi(\delta+\omega)=\{\delta\}$ for every $\delta$ in $C$.
  \item[(ii)] $\varphi(\delta+n)=\{\delta+n, \delta(n)\}$ for every $\delta$ in $C$ and $n \in \omega$, $n>0$.
  \item[(iii)] For all other $\alpha$ in $\omega_1+1$, $\varphi(\alpha)=\{\alpha\}$.   In particular, $\varphi(\delta)=\{\delta\}$ for every $\delta$ in $C$.
\end{enumerate}

Suppose there is a continuous selection $f$.  Note that $f$ is the identity on $C$ (since $\varphi(\delta)=\{\delta\}$, there is no other choice).

Since $\varphi(\delta+\omega)=\{\delta\}$, we must have $f(\delta+\omega)=\delta$.   Then by continuity of $f$ it must be that $f(\delta+k)=\delta(k)$ for some $k$.   Choose such a $k$ and denote it by $k_\delta$.  Then $\delta(k_\delta)=f(\delta+k_\delta)$.

The mapping which sends $\delta$ to $\delta(k_\delta)$ is regressive, so by the pressing down lemma, there is some $\gamma< \omega_1$ such that the set $\{\delta: \delta(k_\delta) = \gamma\}$
is uncountable.  Thus $f^{-1}(\gamma)$ is uncountable.  It is closed too, so club.   But then the function $f$ is both the identity and constant $\gamma$ on the club $C\cap f^{-1}(\gamma)$, which is a contradiction.

It remains to check that $\varphi$ is l.s.c..  Suppose $A\subset \omega_1+1$ and $\alpha\in \varphi(\overline{A})\setminus \varphi(A)$; we need to show $\alpha\in \overline{\varphi(A)}$.
There is some $\beta\in\overline{A}\setminus A$ with $\alpha\in \varphi(\beta)$.

\smallskip

\noindent {\it Case 1. $\beta=\omega_1$.}  Then $A$ is unbounded in $\omega_1$.  It is easy to see that $\varphi(A)$ is also unbounded.   Then, since $\alpha=\omega_1$ in this case, we have $\alpha\in \overline{\varphi(A)}$  .
\smallskip

\noindent {\it Case 2.  $\beta<\omega_1$.}  Let $\{\beta_n:n\in\omega\}$ be an increasing sequence in $A$ converging to $\beta$.

\smallskip

{\it Case 2(a). $\alpha=\beta$.} In this case, if $\beta_n\in \varphi(\beta_n)$ for infinitely many $n$, we are done.  If not, then $\beta_n=\delta_n+\omega$, and so $\varphi(\beta_n)=\{\delta_n\}$,  for all sufficiently large $n$.  But then $\alpha=\beta=sup\{\beta_n:n\in\omega\}= sup\{\delta_n:n\in\omega\}$, so $\alpha\in \overline{\varphi(A)}$.

\smallskip

{\it Case 2(b). $\alpha\neq\beta$.} Then, since $\beta$ is a limit ordinal, for some $\delta\in C$ we must have $\beta=\delta+\omega$ and $\alpha=\delta$.  We then have $\beta_n=\delta+k_n$, $k_n>0$, for all sufficiently large $n$.   Thus $\delta(k_n)\in \varphi(\beta_n)$ for large enough $n$, and $\delta(k_n)$, $n\in\omega$, converges to $\delta=\alpha$.
\end{proof}

We now discuss whether the sequential fan $S_{\omega}$ and Arens' space $S_2$ are self-selective. First recall some definitions. A space $X$ is called an {\it $\alpha_1$-space} \cite{aa} if for every $x\in X$ and any countable family of sequences $\{A_{n}\}_{n\in\omega}$ of sequences converging to $x$
there is a sequence $A$ converging to $x$ such that $|A_{n}\setminus A|<\omega$ for every $n\in\omega$. The {sequential fan} $S_{\omega}$ is the space obtained by
identifying all limit points of the topological sum of $\omega$ many
convergent sequences. Arens' space $S_2=\{x\}\cup\{x_{n}:
n\in\omega\}\cup\{x_{n,m}: n,m\in\omega\}$
is defined as follows: (i) each $x_{n,m}$ is isolated; (ii) a
neighborhood base at $x_{n}$ is formed by the sets
$\{x_{n}\}\cup\{x_{n,m}: m\in\omega\setminus K\}$, where $K$
is finite; (iii) a neighborhood base at $x$ is formed by the sets
$\{x\}\cup\{x_{n}: n\in\omega\setminus K\}\cup\{x_{n,m}:
n\in\omega\setminus K\mbox{ and }m>f(n)\}$ where $K$
is finite and $f: \omega\to\omega$. It is easy to see that $S_{\omega}$ is not $\alpha_1$ and $S_{2}$ is $\alpha_1$ (see Example \ref{s2}).

\begin{prop}\label{alpha_1} Every $L$-selective $T_1$-space is $\alpha_1$.
\end{prop}
\begin{proof} Let $X$ be an $L$-selective $T_1$-space and let $\{A_{n}\}_{n\in\omega}$ be a family of sequences converging to $x\in X$, where $A_{n}=\{x_{n,m}\}_{m\in\omega}$ for every $n\in\omega$. Without loss of generality, we may assume that $x_{n,m}\neq x_{n',m'}$ whenever $(n,m)\neq (n',m')$. Put $P_{n,m}=\{x_{0,n},x_{n,m}\}$ for every $n,m\in\omega$. Then $\mathcal{P}=\{P_{n,m}:n,m\in\omega\}$ is a countable infinite family of nonempty closed  subsets of $X$. Let $\varphi:\omega\to\mathcal{P}$ be a bijection and define $\psi:\omega+1\to\mathcal{F}(X)$ as $\psi (i)=\varphi(i)$ for $i\in\omega$ and $\psi (\omega)=\{x\}$.

We claim that $\psi$ is an l.s.c. mapping. Let $A$  be a infinite subset of $\omega$; it suffices to show that $x\in\overline{\bigcup\varphi (A)}$. If $\{n:P_{n,m}=\varphi (A)$ for some $m\in\omega\}$ is infinite, then $x_{0,n}\in\bigcup\varphi (A)$ for infinite many $n\in\omega$, and thus $x=\lim\limits_{n\to\infty}x_{0,n}\in\overline{\bigcup\varphi (A)}$. If $\{n:P_{n,m}=\varphi (A)$ for some $m\in\omega\}$ is finite, then there exists an $n_{0}\in\omega$ such that $\{m:P_{n_{0},m}\in\varphi (A)\}$ is infinite, and thus $x=\lim\limits_{m\to\infty}x_{n_{0},m}\in\overline{\bigcup\varphi (A)}$. Therefore $\psi$ is a l.s.c. mapping.

By $L$-selectivity, $\psi$ has a continuous selection $f:\omega+1\to X$. Then the sequence $C=\{f(n):n\in \omega\}$ converges to $x$. For every $n\in\omega$, let $B_{n}=\{i:\varphi (i)\in\{P_{n,m}:m\in\omega\}\}$. Obviously $B_{n}$ is a infinite subset of $\omega$, so $f(B_{n})$ is a subsequence of $C$ and thus converges to $x$. Therefore $f(i)\in A_{n}$ for all but finitely many $i\in B_{n}$. Since $\varphi$ is a bijection, then $\varphi|_{B_{n}}:B_{n}\to\{P_{n,m}:m\in\omega\}$ is also a bijection. Thus  $A_{n}\setminus f(B_{n})$ is finite. So $C$ is a sequence converging to $x$ such that $|A_{n}\setminus C|<\omega$ for every $n\in\omega$; therefore $X$ is an $\alpha_1$-space.
\end{proof}

\begin{cor} The {sequential fan} $S_{\omega}$ is not $L$-selective.
\end{cor}

\begin{ex} \label{s2}  Arens' space $S_2$ is $\alpha_1$, but not $L$-selective.
\end{ex}
\begin{proof}
Let $S_2=\{x\}\cup\{x_{n}:
n\in\omega\}\cup\{x_{n,m}: n,m\in\omega\}$ be Arens' space with $\lim\limits_{n\to\infty}x_{n}=x$ and $\lim\limits_{m\to\infty}x_{n,m}=x_{n}$ for every $n\in\omega$.

Arens space is well-known not to be Fr\'echet, so it is not $L$-selective by Corollary~\ref{count_tight}.  To show that $S_2$ is $\alpha_1$, let $\{A_{i}\}_{i\in\omega}$ be a family of sequences converging to a point $y\in S_{2}$. Put
$$A=\left\{
\begin{array}{ll}
\{y\}, & \mbox{ if }y=x_{n,m}\mbox{ for some }n,m\in\omega,\\
\{y\}\cup\{x_{n,m}\}_{m\in\omega}\cap \bigcup_{i\in\omega} A_{i}, & \mbox{ if } y=x_{n} \mbox{ for some }n\in\omega,\\
\{y\}\cup\{x_{n}\}_{n\in\omega}\cap \bigcup_{i\in\omega} A_{i}, & \mbox{ if } y=x.
\end{array}
\right.$$ Then the sequence $A$ converges to $y$ and $|A_{i}\setminus A|<\omega$ for every $i\in\omega$. So we have that $S_2$ is $\alpha_1$.
\end{proof}

In the next section,  assuming  $\mathfrak t=\omega_1$ we will give a Fr\'echet $\alpha_1$-space which is not $L$-selective.

\begin{prop}\label{Limpliesfans} For a space $X$, the following are equivalent:
 \begin{enumerate}
   \item[(1)] $X$ is $L$-selective;
   \item[(2)] $X$ is $S_2$-selective;
   \item[(3)] $X$ is $S_{\omega}$-selective.
     \end{enumerate}
 \end{prop}
 \begin{proof} (2)$\Rightarrow$(1) and (3)$\Rightarrow$(1) follow from Lemma~\ref{c_subs_d}.

 (1)$\Rightarrow$(2). Let $S_2=\{x\}\cup\{x_{n}:
n\in\omega\}\cup\{x_{n,m}: n,m\in\omega\}$ be the Arens' space with $\lim\limits_{n\to\infty}x_{n}=x$ and $\lim\limits_{m\to\infty}x_{n,m}=x_{n}$ for every $n\in\omega$, and $\varphi:S_{2}\to\mathcal{F}(X)$ be a l.s.c mapping. Put $L=\{x\}\cup\{x_{n}:
n\in\omega\}$ and $L_{n}=\{x_{n}\}\cup\{x_{n,m}: m\in\omega\}$ for every $n\in\omega$. Fix $y\in\varphi (x)$ and define $\psi:L\to\mathcal{F}(X)$ as $\psi(x)=\{y\}$ and $\psi(x_{n})=\varphi(x_{n})$. It is easy to verify that $\psi$ is an l.s.c. mapping from $L$ to $\mathcal{F}(X)$. Since $X$ is $L$-selective, $\psi$ has a continuous selection $g:L\to X$.
Now for every $n\in\omega$, define $\psi_{n}:L_{n}\to\mathcal{F}(X)$ as $\psi_{n}(x_{n})=\{g(x_{n})\}$ and $\psi_{n}(x_{n,m})=\varphi(x_{n,m})$ for every $m\in\omega$. Then every $\psi_{n}$ is an l.s.c. mapping and thus has a continuous selection $g_{n}:L_{n}\to X$. Define $f:S_{x}\to X$ as $f(z)=g(z)$ for $z\in L$ and $f(z)=g_{n}(z)$ for $z\in L_{n}$. Obviously $f$ is a selection of $\varphi$. It is easy to verify that for every convergent sequence $S$ in $S_{2}$, $f(S)$ is also a convergent sequence in $X$. Since $S_2$ is a sequential space,  we have that $f$ is continuous. So $X$ is $S_2$-selective.

 (1)$\Rightarrow$(3). Let $S_{\omega}=\{x\}\cup\{x_{n,m}: n,m\in\omega\}$ be the sequential fan with $\lim\limits_{m\to\infty}x_{n,m}=x$ for every $n\in\omega$, and $\varphi:S_{\omega}\to\mathcal{F}(X)$ be a l.s.c mapping. Put  $L_{n}=\{x\}\cup\{x_{n,m}: m\in\omega\}$ for every $n\in\omega$. Fix $y\in\varphi (x)$ and define $\psi_{n}:L_{n}\to\mathcal{F}(X)$  as $\psi_{n}(x)=\{y\}$ and $\psi_{n}(x_{n,m})=\varphi(x_{n,m})$ for every $n,m\in\omega$. Then every $\psi_{n}$ is an l.s.c. mapping and thus has a continuous selection $g_{n}:L_{n}\to X$. Define $f:S_{\omega}\to X$ as $f(x)=y$ and $f(z)=g_{n}(z)$ for $z\in L_{n}$. Then $f$ is a continuous selection of $\varphi$.
  So $X$ is $S_{\omega}$-selective.
 \end{proof}

By the same idea, one may show that  $L$-selectivity is equivalent to $S_{\kappa}$-selectivity for any $\kappa\geq\omega$.

\begin{cor} Neither $S_{\omega}$ nor $S_2$ is self-selective.
\end{cor}

\section{on $C$-selective spaces}\label{c-sel}
Let us recall the definition of the W-spaces, which was introduced in \cite{gg}. Let $x$ be a point in the topological space $X$, and consider the following two-person infinite game: player I chooses an open set $U_{0}$ containing $x$, and then player II chooses a point $x_{0}\in U_{0}$; player I then chooses another open set $U_1$ containing $x$, player II chooses some point $x_{1}\in U_{1}$, and so on. We say that player I {\it wins} the game if the sequence $\langle x_1,x_2,\ldots\rangle$ converges to $x$. $X$ is called a {\it W-space} if there exists a winning strategy for player I at each point of $X$. Obviously, every first-countable space is a $W$-space. E. Michael proved in \cite{m81} that every first-countable space is strongly $C$-selective; we show that $W$-spaces are strongly $C$-selective too.  Since $W$-spaces are closed under $\Sigma$-products and subspaces, it follows that, e.g., Corson compact spaces are strongly $C$-selective.

\begin{prop} \label{wimplycsel} Every $W$-space is strongly $C$-selective. \end{prop}

\begin{proof} Let $C$ be a countable regular space,  and let $\varphi$ be an l.s.c. mapping from $C$ to the nonempty subsets of the $W$-space $X$.  Let $\sigma^x$ denote a winning strategy for the open set picker in the  $W$-game played at point $x\in X$; i.e., $\sigma^x:X^{<\omega}\rightarrow \{\textrm{open neighborhoods of }x\}$ such that for each sequence $x_0,x_1,...$ of points of $X$:
$$\textrm{If }x_{i+1}\in \sigma^x(\langle x_0,x_1,...,x_i\rangle)\textrm{ for each }i,\textrm{ then } x_i\rightarrow x.$$

Let $Y_0$ be a countable subset of $X$ such that $Y_0\cap \varphi (c)\neq \emptyset$ for each $c\in C$.  Then let $Y_1\supset Y_0$ be countable and satisfy:
$$\forall c\in C\forall y\in Y_0 \forall s\in Y_0^{<\omega} (\varphi (c)\cap \sigma^{y}(s)\neq \emptyset \imp \varphi (c)\cap \sigma^{y}(s)\cap Y_1\neq \emptyset).$$

If $Y_n$ has been defined, let $Y_{n+1}\supset Y_n$ be countable and satisfy:
$$\forall c\in C\forall y\in Y_n \forall s\in Y_n^{<\omega} (\varphi (c)\cap \sigma^{y}(s)\neq \emptyset \imp \varphi (c)\cap \sigma^{y}(s)\cap Y_{n+1}\neq \emptyset).$$

Let $Y=\bigcup_{n<\in\omega}Y_n$.  Then $Y$ is countable and satisfies:

$$\forall c\in C\forall y\in Y \forall s\in Y^{<\omega} (\varphi (c)\cap \sigma^{y}(s)\neq \emptyset \iff \varphi (c)\cap \sigma^{y}(s)\cap Y\neq \emptyset).$$

 Since $Y$ is a countable $W$-space, it is first countable.  Indeed, by \cite[Theorem 3.3]{gg}, for each $y\in Y$, the set
 $$\{\sigma^y(s)\cap Y: s\in Y^{<\omega}\}$$ is a (countable) base at $y$ in the subspace $Y$, and so $$\mathcal B= \{\sigma^y(s)\cap Y: y\in Y, s\in Y^{<\omega}\}$$
 is a countable base for $Y$.

 For each $c\in C$, let $\varphi_Y(c)=\varphi(c)\cap Y$. We claim that $\varphi_Y$ is an l.s.c. mapping from $C$ to the nonempty subsets of $Y$.  It suffices to check, for each $c\in C$ and $B=\sigma^y(s)\cap Y\in \mathcal B$, that $\varphi_Y^{-1}(B)=\varphi^{-1}(\sigma^y(s))$; equivalently, that $$\varphi(c)\cap \sigma^y(s)\neq \emptyset \text{ iff }  \varphi_Y(c)\cap(\sigma^y(s)\cap Y) \neq \emptyset.$$
 The reverse direction is trivial, so suppose $\varphi(c)\cap \sigma^y(s)\neq \emptyset$.  Since $y\in Y$, $y\in Y_n$ for some $n$, and hence $\varphi (c)\cap \sigma^{y}(s)\cap Y_{n+1}\neq \emptyset.$  So $\varphi_Y(c)\cap (\sigma^y(s)\cap Y)=\varphi(c)\cap \sigma^y(s)\cap Y\neq \emptyset$.
   This proves the claim.

 By Michael's theorem, $\varphi_Y$ has a continuous selection $f:C\rightarrow Y$, and such an $f$ is clearly a continuous selection for $\varphi$ as well.
\end{proof}

\begin{prop}\label{prod_go_ctbl} If $X$ is a GO-space and $E$ is a metrizable countable space, then $E\times X$  is $C$-selective. \end{prop}

\begin{proof} Let $C$ be a countable regular space,  and let $\varphi$ be an l.s.c. mapping from $C$ to the nonempty closed subsets of the product space $E\times X$.
Given $x\in X$, if the interval $(-\infty,x)$ has $x$ in its closure and has countable cofinality, let $l_n(x), n<\omega $, be an increasing sequence of points converging to $x$; define $r_n(x)$ analogously.  Now define $B_2(x,n)$ as follows:
\begin{enumerate}
  \item[(i)] Let $B_2(x,n)=(l_n(x),r_n(x))$ if both $l_n(x)$ and $r_n(x)$ are defined;
  \item[(ii)]Let $B_2(x,n)=(l_n(x),x]$ if  $l_n(x)$ is defined and $r_n(x)$ is not; if vice-versa, let $B(x,n)=[x,r_n(x))$;
  \item[(iii)]Let $B_2(x,n)=\{x\}$ otherwise.
\end{enumerate}
Let $\mathcal{B}_2(x)=\{B_2(x, n): n\in \omega\}$. For any $z\in E$,we fix a countable local base $\mathcal{B}_1(z)$ at $z$. Then we define $\mathcal{B}(z, x)=\{B_1\times B_2: B_1\in \mathcal{B}_1(z)\text{ and } B_2\in\mathcal{B}_2(x)\}$ which is clearly countable.

Let $Y_0$ be a countable subset of $E\times X$ such that $Y_0\cap \varphi (c)\neq \emptyset$ for each $c\in C$. If $Y_n$ has been defined, let $Y_{n+1}\supset Y_n$ be countable and satisfy:
$$\forall c\in C\forall y\in Y_n \forall B\in \mathcal{B} (\varphi (c)\cap B\neq \emptyset \imp \varphi (c)\cap B\cap Y_{n+1}\neq \emptyset).$$

Let $Y=\bigcup_{n<\in\omega}Y_n$.  Then $Y$ is countable and satisfies:
$$\forall c\in C\forall y\in Y \forall B\in \mathcal{B}(\varphi (c)\cap B\neq \emptyset \iff \varphi (c)\cap B\cap Y\neq \emptyset).$$

 First, we claim that for each $y=(y_1, y_2)\in Y$, $\{B\cap Y: B\in\mathcal{B}(y)\}$ is a base at $y$ in the subspace $Y$ (so in particular $Y$ is first countable). Fix $y=(y_1, y_2)\in Y$.  If both $l_n(y_2)$ and $r_n(y_2)$ are defined as in case (i) above, it is clear that $\mathcal{B}(y)$ is a countable local base at $y$. Now we consider the case (ii) above where $B_2(y_2,n)=(l_n(y_2),y_2]$ because  $l_n(y_2)$ is defined and $r_n(y_2)$ is not.  Since $r_n(y)$ is not defined, then either $y$ has a base in  $X$ of left half-open intervals, or the coinitiality of $(y,\infty) $ is uncountable.  In either case, since $Y$ is countable, $y_2$ has a base in $\pi_2(Y)$ of left half-open intervals. Hence $\{B\cap Y: B\in\mathcal{B}(y)\}$ is a base at $y$ in the subspace$Y$  in this case.  We leave the straightforward checking of the other cases to the reader.

For each $c\in C$, let $\varphi_Y(c)=\varphi(c)\cap Y$. We claim that $\varphi_Y$ is an l.s.c. mapping from $C$ to the closed subsets of $Y$.  Since $\{B\cap Y: B\in\mathcal{B}(y)\}$ is a base at $y$ in the subspace $Y$ for any $y\in Y$, it suffices to show that for each $B\in\mathcal{B}(y)$,$\varphi_Y^{-1}(B\cap Y)$ is open in $C$.  By construction, for each $B\in \mathcal{B}(y)$ we have $\varphi_Y(c)\cap B\cap Y= \varphi(c) \cap B\cap Y$ and  $\varphi(c) \cap B\cap Y\neq \emptyset$
iff $\varphi(c) \cap B\neq \emptyset$.  It follows that $\varphi_Y^{-1}(B\cap Y)=\varphi^{-1}(B)$, so we just need to     prove that for each $B\in\mathcal{B}(y)$, $\varphi^{-1}(B)$ is open.

Fix $y\in Y$ and $B=B_1\times B_2\in\mathcal{B}(y)$. The result clearly holds if $B_2$ is open in $X$, e.g., as in case (i) above.

Suppose $B_2=(l_n(y_2),y_2]$ for some $n\in\omega$ as in case (ii). If $B_2$ is not open in $X$, then $(y_2,\infty)$ has $y_2$ in its closure and has uncountable coinitiality. Fix $\{y_{2\alpha}: \alpha<\kappa\}$ where $\kappa$ is an uncountable cardinal such that $y_{2\alpha}<y_{2\alpha'}$ given $\alpha>\alpha'$ and $y_2\in \overline{\{y_{2\alpha}: \alpha<\kappa\}}$.  Let $A=\{c\in C: \varphi(c)\cap B=\emptyset\}$.
Fix $c\in A$. We claim that there is a $\alpha_c$ such that $B_1\times (l_n(y_2), y_{2\alpha_c})\cap \varphi(c)=\emptyset$. Otherwise, there exists $t_\alpha\in B_1\times (l_n(y_2), y_{2\alpha})$ for each $\alpha\in \kappa$. Since $E$ is countable and $\kappa$ is uncountable, there is a cofinal subset $D$ of $\kappa$ such that $\pi_1(t_\gamma)$'s are same, denoted by $s$, for all $\gamma\in D$.  Then $\{s\}\times\{ y_2\}\in \varphi(c)\cap B$ which is impossible since $c\in A$. Then we let $\beta=\sup\{\alpha_c: c\in A\}$ which is clearly $<\kappa$ and let $r_{y_2}=y_{2\beta}$. Also, $(B_1\times (l_n(y_2), r_{y_{2}}))\cap  \varphi(c)=\emptyset$ for each $c\in A$, i.e. $\varphi^{-1}(B)=\varphi^{-1}(B_1\times (l_n(y_2), r_{y_{2}}))$ which is open since $\varphi$ is l.s.c.  Obviously a similar argument works if $B(y,n)=[y,r_n(y))$.


It remains to prove that $\varphi^{-1}(B)$ is open when $B_2=\{y_2\}$.  This is obvious when $y_2$ is an isolated point in $X$.  The remaining cases are handled similarly to case (ii) by defining $l_{y_2}$ and/or $r_{y_2}$ as in that case if appropriate, and then see that $\varphi^{-1}(B_1\times\{y_2\})$ is equal to the preimage of the open set $B_1\times (l_{y_2},r_{y_2})$, $B_1\times(l_{y_2},y_2]$, or $B_1\times[y_2,r_{y_2})$.

Finally, the proof is finished by an application of Michael's theorem as in the previous proposition.
\end{proof}

\begin{cor}Any GO-space is $C$-selective.

\end{cor}

There are GO-spaces which are not strongly $C$-selective. Since any strongly $L$-selective space is Frech\'et by Proposition~\ref{sl_f}, the space $\omega_1+1$ is clearly not strongly $L$-selective, hence not strongly $C$-selective.

In the next example, we define two GO-spaces whose product is not $L$-selective.  This shows that the property of being $C$-selective is not productive.
\begin{ex} There are two GO-spaces whose product is not $L$-selective.

\end{ex}

\begin{proof} Let $X=  (\omega_1\cdot \omega+1)\setminus \{\omega_1\cdot n: n\in \omega\}$  and $Y=\omega_1+1$. Define a mapping from $\omega+1$ to $X\times Y$ as follows: $\varphi(n)=\{ (\omega_1\cdot n+\alpha, \alpha): \alpha\in \omega_1\}$ and $\varphi(\infty)=\{(\omega_1\cdot \omega, \omega_1)\}$. For each $n\in\omega$, $\varphi(n)$ is closed in $X\times Y$ since $\omega_1\cdot n\notin X$.

We claim that $\varphi$ is l.s.c. It is enough to show that $\varphi^{-1}(U)$ is open in $\omega+1$ for any open neighborhood  $U$ of $(\omega_1\cdot \omega, \omega_1)$. Let $U$ be an open neighborhood of  $(\omega_1\cdot \omega, \omega_1)$. Then there exist $n\in \omega$ and $\alpha\in\omega_1$ such that $(\{\beta: \omega_1\cdot n\le \beta\le \omega_1\cdot\omega+1\}\times \{\beta: \alpha\leq \beta\leq \omega_1\})\cap (X\times Y)\subseteq U$. Then $(\omega_1\cdot m+\gamma, \gamma)\in \varphi(m)\cap U$ for any $m\geq n$ and $\gamma>\alpha$. Hence $\varphi^{-1}(U)=(\omega+1)\setminus n$ which is clearly open. Therefore, $\varphi$ is l.s.c.

Then for any selection $f$ with $f(n)\in \varphi(n)$, $\sup\{\pi_2(f(n)): n<\omega\}<\omega_1$, hence $f$ can't be continuous. Hence $\varphi$ has no continuous selection. Therefore, $X\times Y$ is not $L$-selective. \end{proof}

The following example shows that the property of being $C$-selective is not hereditary.

\begin{ex}\label{subspace} There is a compact $C$-selective space which has a non-$L$-selective subspace.

\end{ex}

\begin{proof} Let $X=(\omega_1+1)\times (\omega+1)$; this space is $C$-selective by Proposition~\ref{prod_go_ctbl}. Let $Y=(\omega_1\times\omega)\cup \{(\omega_1, \omega)\}$. Define $\varphi$ as follows: $\varphi(\omega)=\{(\omega_1, \omega)\}$ and $\varphi(n)=\omega_1\times\{n\}$ for each $n\in\omega$. It is straightforward to check that $\varphi$ is l.s.c. and has no continuous selection.
\end{proof}

A property $W$-spaces and GO-spaces have in common is that every countable subset is first countable; we denote this property by $CFC$. Observe that it is consistent with ZFC that $L$-selective spaces are $CFC$.  This is because every countable subspace of an $L$-selective space is $\alpha_1$ (by Proposition~\ref{alpha_1}) and Fr\'echet (by Corollary~\ref{count_subspace}), and Dow and Steprans \cite{ds} have shown that it is consistent that every countable Fr\'echet $\alpha_1$-space is first countable. The $CFC$ property was important in the proofs that $W$-spaces and GO-spaces are $C$-selective, but also important was having a definable way in which to describe the countable base at a point in a countable subset.  Indeed, the subspace $Y$ of Example~\ref{subspace} shows that $CFC$ alone does not imply $L$-selective.

It turns out that Fr\'echet $CFC$ (which $W$-spaces satisfy) also does not imply $L$-selective, and least consistently.  Here we present an example of a Fr\'echet $CFC$ space which is not $L$-selective assuming  $\mathfrak t=\omega_1$, an axiom weaker than CH.  One may do essentially the same construction for any value of $\mathfrak t$ and the resulting space will be $CFC$, but it need not be Fr\'echet.  We don't know an example of a Fr\'echet $CFC$-space which is not $L$-selective (or even not $C$-selective) in ZFC.

\begin{ex}\upshape{($\mathfrak t=\omega_1$)} There exists a Fr\'echet $CFC$ space which is not $L$-selective.
\end{ex}
\begin{proof}Let $A_\alpha$, $\alpha<\kappa$, be an almost decreasing sequence of subsets of $\omega$ (i.e., $\alpha<\beta$ implies $A_\beta\subset^*A_\alpha$) with no infinite pseudo-intersection (i.e., there is no infinite $A$ such that $A\subset^*A_\alpha$ for all $\alpha<\kappa$).  Such a family is called a {\it tower}, and $\mathfrak t$ denotes
  the least cardinality of a tower.  It is easy to see that there always exists a tower of some (uncountable) length; $\mathfrak t=\omega_1$ says there is one of length $\omega_1$.

Let $X=(\omega_1\times\omega)\cup\{\infty\}$.  For each $\alpha<\kappa$, let $$H_\alpha=\{(\beta,n)\in X: \beta\leq \alpha\textrm{ and }n\in A_\alpha\}$$ and let $U_\alpha=X\setminus H_\alpha$.  Now let $B(\infty,n)=X\setminus (\omega_1\times n)$, and finally let $\tau$ be the topology on $X$ generated by
$$\{B(\infty,n):n\in\omega\}\cup\{U_\alpha:\alpha<\kappa\}\cup\{\{x\}:x\in \kappa\times\omega\}.$$  So $\infty$ is the only non-isolated point, and it has a local basis of sets of the form $B(\infty,n)\cap\bigcap_{\alpha\in F}U_\alpha$, where $n\in\omega$ and $F$ is a finite subset of $\kappa$.
We will show that $(X,\tau)$ is $CFC$ but not $L$-selective, and if $\kappa=\omega_1$ it is also Fr\'echet.

Let $L_n=\omega_1\times\{n\}$ and define a map $\varphi$ from $\omega+1$ to the closed subsets of $X$ by $\varphi(n)=L_n$ for each $n\in\omega$ and
$\varphi(\omega)=\{\infty\}$.  Note that for each $\alpha<\kappa$, $U_\alpha\cap L_n=\{(\beta,n)\in X: \beta>\alpha\}$, and hence a basic open set of the form
$B(\infty,n)\cap\bigcap_{\alpha\in F}U_\alpha$, where $F$ is finite, meets $L_m$ for all $m\geq n$.  It easily follows that $\varphi$ is l.s.c.

We will show that there is no continuous selection for $\varphi$.  Suppose $f:\omega+1\rightarrow X$ with $f(\omega)=\infty$ and $f(n)\in L_n$ for each $n$.
For each $n\in\omega$, let $\alpha_n<\kappa$ be such that $f(n)=(\alpha_n,n)$.  Pick $\alpha<\kappa$ with $\alpha>\alpha_n$ for all $n$.  If $n\in A_\alpha$, note that
$(\alpha_n,n)\in H_\alpha$ and hence $f(n)\not\in U_\alpha$.  Thus $f^{-1}(U_\alpha)=\omega+1\setminus A_\alpha$ which is not open in $\omega+1$.  Hence $f$ is not continuous, and so $X$ is not $L$-selective.

Next we prove that $X$ is $CFC$.  It suffices to show that every subspace of $X$ of the form $Y=(C\times\omega)\cup\{\infty\}$ is first countable, where $C$ is a countable subset of $\kappa$.
Pick such a $C$, and let $\overline{C}$ denote the closure of $C$ in the ordinal space $\kappa$.  Note that $\overline{C}$ is also countable. We claim that
$$\mathcal{B}=\{Y\cap B(\infty,n)\cap\bigcap_{\alpha\in F}U_\alpha:n\in\omega, F\in[\overline{C}]^{<\omega}\}$$
is a countable base for the subspace $Y$.  Since $\mathcal B$ is closed under finite intersections, it will suffice to show that each $U_\delta$,
$\delta\in\kappa\setminus \overline{C}$, contains some member of $\mathcal B$.  Fix such a $\delta$.  If $\delta<\textrm{min}(C)$, then $U_\delta\supset Y$, so we may assume $\delta\cap C\neq \emptyset$.  Let $\gamma$ be the maximal element of $\overline{C}$ below $\delta$.  Since $A_\delta\subset^*A_\gamma$, the set $F=A_\delta\setminus A_\gamma$
is finite.  Let $n>\textrm{max}(F)$, and suppose $(\beta,m)\in Y\cap B(\infty,n)\cap U_\gamma$.  If $\beta>\delta$, then $(\beta,m)\in U_\delta$, so we may assume
$\beta<\delta$.  Then $\beta\leq\gamma$.  Note that $m\not\in F\cup A_\gamma$, whence $m\not\in A_\delta$ and so $(\beta,m)\in U_\delta$.  Thus $U_\delta$ contains
$Y\cap B(\infty,n)\cap U_\gamma$, a member of $\mathcal B$, and we conclude that $\mathcal B$ is a countable base for $Y$.

Note that up to now we have not used the assumption that $\{A_\alpha:\alpha<\kappa\}$ has no infinite pseudo-intersection.  It will be needed for the proof of Fr\'echet.
Suppose $\kappa=\omega_1$, and that $\infty\in \overline{Y}\setminus Y$ for some $Y\subset X$.  Since $X$ is $CFC$, it will suffice to show that $\infty\in \overline{Z}$ for some
countable $Z\subset Y$.  Let $$A=\{n\in\omega:Y\cap L_n\textrm{ is uncountable}\}.$$ If $A$ is finite, then $\infty$ is in the closure of the countable set $Y\setminus\bigcup_{n\in A}L_n$.  Thus we may assume $A$ is infinite.  Since $\{A_\alpha:\alpha<\kappa\}$ has no infinite pseudo-intersection, we may choose
$\delta<\kappa$ such that $B=A\setminus A_\delta$ is infinite.  For each $n\in B$, choose $\alpha_n>\delta$ with $(\alpha_n,n)\in Y$.  We claim that $Z=\{(\alpha_n,n):n\in B\}$
is sequence converging to $\infty$.  It will suffice to show that any subbasic open neighborhood of $\infty$ contains all but finitely many points of $Z$.  For a neighborhood
of the form $B(\infty,n)$ this is obvious.  Now consider some $U_\gamma$ for $\gamma\in \kappa$.  If $\gamma\leq \delta$, then since $\alpha_n>\gamma$ for all $n\in B$, we
have $U_\gamma\supset Z$.  Suppose then that $\gamma>\delta$.  Since $A_\gamma\subset^*A_\delta$, we have $A_\gamma\cap B$ is finite.  Hence $U_\gamma\cap Z$ is finite, which concludes the proof.
\end{proof}

The following example of a Fr\'echet $CFC$-space which is strongly $C$-selective but not a $W$-space  gives some hope for extending our results a little bit.  The example is an old example of Hajnal and Juhasz (see \cite{gn}) of a Fr\'echet
$CFC$-space which is not a $W$-space, and can be simply described as the one-point compactification of an Aronszajn tree with the interval topology.

\begin{ex}  The one-point compactification of an Aronszajn tree with the interval topology is a Fr\'echet $CFC$-space which is strongly $C$-selective but is not a $W$-space.
\end{ex}

\begin{proof}
 Let $T$ be an Aronszajn tree, i.e., its height is $\omega_1$ and each chain and each level is countable.   We also assume that if $s$ and $t$ are nodes at the same limit level and have the same set of predecessors, then $s=t$.  The interval topology on $T$ is generated by sets
of the form $$(s,t]=\{x\in T: s<x\leq t\}$$ where $s$ and $t$ are elements of the tree with $s<t$.   Then  $T$ with the interval topology is a locally compact Hausdorff space, so it has a one-point compactification $X=T\cup\{\infty\}$.

We observe the following properties of $X$.  For each $t\in T$, let $K_t=\{s\in T:s\leq t\}$, and for $F\in [T]^{<\omega}$ let $K_F=\bigcup_{t\in F}K_t$.  Then every compact subset of $T$ is contained in some such $K_F$, and so $\{X\setminus K_F: F\in [T]^{<\omega}\}$ is an open neighborhood base at $\infty$. In particular, every neighborhood of $\infty$ is co-countable.  Hence if $H$ is a closed subset of $X$, then either $\infty\in H$ or $H$ is countable.  In the latter case, there is some $\delta<\omega_1$ such that $H\subset T_{\leq \delta}$, where   $T_{\leq \delta}$ is the set of all nodes of $T$ of level not greater than $\delta$.

To see that every countable subspace of $X$ is first countable, it suffices to show that $\infty$ has countable character in every subspace of the form $Y=T_{\leq \delta}\cup\{\infty\}$.  Consider a basic open neighborhood $X\setminus K_F$ of $\infty$, where $F$ is a finite subset of $T$. Then
$(X\setminus K_F)\cap Y=(X\setminus K_{F'})\cap Y$, where $F'$ is obtained from $F$ by replacing each $t\in F$ at a level above $\delta$ with its unique predecessor
at level $\delta$.  Then $F'$ is a finite subset of $T_{\leq \delta}$, and there are only countably many such $F'$, so $\infty$ is a point of first countability in
$Y$.  Thus $Y$ is first countable.

See \cite{gn} for the proof that $X$ is not a $W$-space.

It remains to prove that $X$ is strongly $C$-selective.  The following claim will be useful.

{\it Claim. Suppose $A$ is an uncountable subset of $T$.  Then there is $\delta<\omega_1$ such that $\infty\in \cl{A\cap T_{\leq \delta}}$.}
To see this, first note that any antichain of $T$ is closed discrete in $T$ and thus any infinite antichain has $\infty$ in its closure.  Now any uncountable subset
$A$ of $T$ must contain an infinite antichain.  One way to see this is to note that $A$ as a subtree has no uncountable chains, and hence some level
(and so some antichain) must be infinite.  If $A'\subset A$ is a countably infinite antichain, then any $\delta$ such that $A'\subset T_{\leq \d}$ satisfies
the conclusion of the claim.

Now suppose $C$ is a countable regular space and that $\varphi$ is an l.s.c. map from $C$ to the nonempty subsets of $X$.  Let $\delta<\omega_1$ be such that
$\varphi(c)\subset T_{\leq \delta}\cup\{\infty\}$ whenever $\varphi(c)$ is countable, and $\infty\in \cl{\varphi(c)\cap T_{\leq \d}}$ if $\varphi(c)$ is uncountable.
 Let $Y=T_{\leq \delta}\cup\{\infty\}$ and for each $c\in C$ let $\psi(c)=\varphi(c)\cap Y$.

We will show that $\psi$ is an l.s.c. mapping from $C$ to the nonempty  subsets of $Y$.  Suppose $U$ is  relatively open  in $Y$.  As $U\subset Y$, we have
$\psi(c)\cap U\neq \emptyset $ iff $\varphi(c)\cap U\neq \emptyset$ and hence $\psi^{-1}(U)=\varphi^{-1}(U)$.  If $\infty\not\in U$, then since $T_{\leq \delta}$ is open
in $X$, $U$ is open in $X$ and thus $\psi^{-1}(U)$ is open in $C$.  Finally, suppose $\infty \in U$.  Let $U^*$ be open in $X$ such that $U^*\cap Y=U$. Suppose $\varphi(c)\cap U^*\neq \emptyset$. If $\varphi(c)$ is countable, then it is contained in $Y$ and so $\varphi(c)\cap U\neq \0$.  If $\varphi(c)$ is uncountable, then by choice of $\d$,
$\infty\in \cl{\varphi(c)\cap T_{\leq \d}}$, and hence $\varphi(c)\cap U=\varphi(c)\cap (U^*\cap T_{\leq \d}))=U^*\cap (\varphi(c)\cap T_{\leq \d})\neq \0$.  It follows that $\varphi^{-1}(U)=\varphi^{-1}(U^*)$.  But $\psi^{-1}(U)=\varphi^{-1}(U)$, so $\psi^{-1}(U)$ is open.

 Since $\psi$ is l.s.c. and $Y$ is first countable, there is a continuous selection $f:C\rightarrow Y$.  Then $f$  is a continuous selection
 for $\varphi$ as well.
\end{proof}


\section{ $L$-selective spaces and $\mathbb Q$-selective spaces}\label{l-sel-nq-sel}
 Assuming $\mathfrak p=\mathfrak c$, we will construct a countable $L$-selective space $X$ which is not $\mathbb Q$-selective, where $\mathbb Q$ is the space of rational numbers, Since $L$-selective spaces are Fr\'echet $\alpha_1$ and it is consistent that countable Fr\'echet $\alpha_1$-spaces are first-countable, there cannot be a ZFC example like this.  We don't know if there is a ZFC example where $X$ is uncountable.

Let $M_\omega$ denote the countable metric fan; i.e., $M_\omega=(\omega\times \omega)\cup\{\infty\}$ where the points of $\omega\times\omega$ are isolated and a basic open set about $\infty$ is $X\setminus (n\times \omega)$ where $n\in\omega$.  Instead of constructing our example so that it is not $\mathbb Q$-selective directly, it will be convenient to construct
it so that it is not $M_\omega$-selective, as $M_\omega$ is a much simpler space.   The next couple of results show that this will suffice.

The next proposition is surely known (possibly folklore).

\begin{prop} Let $M$ be a countable metrizable space.  Then $M$ is homeomorphic to a closed subspace of $\mathbb Q$. \end{prop}

\begin{proof}  Let $Y$ be $M$ with each isolated point of $M$ replaced by a copy of $\mathbb Q$.  Then $Y$ is a countable metrizable space with no isolated points.  By a classical theorem of Sierpinski, $Y$ is homeomorphic to $\mathbb Q$.  Let $Z$ be the subspace of $Y$ obtained by throwing out of $Y$ all points in each copy of $\mathbb Q$ except one, say the point 0.  Clearly $Z$ is closed in $Y$, and $M$ is homeomorphic to $Z$,   \end{proof}

The following corollary is immediate from the previous proposition and Proposition~\ref{c_subs_d}.

\begin{cor}\label{QimpliesY} A $\mathbb Q$-selective space is $Y$-selective for any countable metric space $Y$. \end{cor}

For the proof of the following result, recall that $\mathfrak p=\mathfrak c$ is equivalent to MA($\sigma$-centered); i.e., given a collection $\mathcal D$ of fewer than $\mathfrak c$-many dense subsets of a $\sigma$-centered poset $\mathbb P$, there is a filter $G$ in $\mathbb P$ which meets every member of $\mathcal D$.

\begin{ex}\label{notMomega}\upshape{($\mathfrak p=\mathfrak c$)} There is a countable  strongly $L$-selective space which is not $\mathbb Q$-selective.\end{ex}

\begin{proof} We will construct a countable space $X$ which is strongly $L$-selective but not $M_\omega$-selective.  By Corollary~\ref{QimpliesY}, $X$ is not $\mathbb Q$-selective either.

The set for $X$ is $\omega^3\cup\{\infty\}$.  Let $\tau_0$ be the topology on $X$ such that points of $\omega^3$ are isolated, and a local base at $\infty$ is $\{X\setminus ( n\times \omega^2 ): n\in\omega\}$.  We are thinking of $(X,\tau_0)$ as the metric fan with each isolated point replaced by a copy of $\omega$ (in fact, it is homeomorphic to the metric fan). The plan is to modify the topology of $(X,\tau_0)$ by adding open sets to $\tau_0$ in an induction of length $\mathfrak c$ to make it strongly $L$-selective but not $M_\omega$-selective.

Define $\psi:M_\omega\rightarrow {\mathcal P}(X)\setminus \{\emptyset\}$ by $\psi(\infty)=\{\infty\}$ and $\psi(m,n)=L_{mn}$, where $L_{mn}=\{m\}\times\{n\}\times\omega$.  When finished with our induction, we want this $\psi$ to be l.s.c.  but have no continuous selection.  To this end, let $\{f_\alpha:\alpha<\mathfrak c\}$ list all functions $f:M_\omega\rightarrow X$ such that $f(\infty)=\infty$ and $f(m,n)\in L_{mn}$ for each $m, n\in\omega$.  We also let $\{\varphi_\alpha: \alpha<\mathfrak c\}$ list all functions $\varphi:\omega+1\rightarrow {\mathcal P}(X)\setminus \{\emptyset\}$ such that $\varphi(\omega)=\{\infty\}$. 
Since $X$ is countable, $\mathcal P(X)$ and $\mathcal{P}(X)^\omega$ have cardinality continuum $\mathfrak c$, so such a listing is possible.

At stage $\alpha$ of the construction, we will first add an open set to the topology of $X$ which makes $f_\alpha$ not continuous.  We then look at $\varphi_\alpha$.  If $\varphi_\alpha$ is not l.s.c. in the topology as defined so far, we do nothing.  If it is, we specify a continuous selection $g_\alpha$ with the intention of keeping it continuous through all later stages.

 Suppose $\alpha<\mathfrak{c}$ and for each $\beta<\alpha$ we have constructed $U_\beta\subset X$, $g_\beta:\omega+1\rightarrow X$, and topology $\tau_\beta$ on $X$ satisfying:

 \begin{enumerate}
   \item[(i)] $\tau_\beta$ is the topology generated by $\tau_0\cup\{U_\gamma:\gamma<\beta\}$;
   \item[(ii)] $U_\beta=X\setminus \{f_\beta(x_i):i\in\omega\}$, where $\{x_i:i\in\omega\}$ is a sequence of points of $M_\omega$ converging to $\infty$;
   \item[(iii)] If $\gamma<\beta$, then $\{f_\beta(x_i):i\in\omega\}\cap \{g_\gamma(n):n\in\omega\}$ is finite.
   \item[(iv)] If $\varphi_\beta$ is closed l.s.c. when $X$ has the topology generated by $\tau_\beta\cup \{U_\beta\}$, then $g_\beta$ is a continuous selection; otherwise, $g_\beta$ is constant $\infty$.
  \end{enumerate}

 We check that (i)--(iv) can be had at step $\alpha$.  Note that $\tau_\alpha$ is given by (i). Let $\mathcal{B}$ be the collection of finite subsets of $\{x_i: x_i\in \{i\}\times\omega \text{ and }i\in \omega\}$ such that for any $s\in\mathcal{B}$, $i\neq i'$ if $x_i, x_{i'}\in s$.  For (ii) and (iii), define $\mathbb{P}_1 =\{\langle s, F\rangle:F \text{ is a finite subset of } \{g_\gamma: \gamma<\alpha\}\text{ and } s\in \mathcal{B}\}$. Then we define  an order $\leq_1$ on $\mathbb{P}_1$ as follows:
   $$\langle s, F\rangle\leq_1 \langle s', F'\rangle \leftrightarrow s\supset s', F\supset F', \forall\gamma\in F'(\{g_\gamma(n): n\in \omega\} \cap \{f_\alpha(x_i): x_i\in s\}\subseteq \{f_\alpha(x_i): x_i\in s'\}). $$

 Notice that for any $\langle s, F_1\rangle$, $\langle s, F_2\rangle\in (\mathbb{P}_1, \leq_1)$, there is an extension $\langle s, F_1\cup F_2\rangle$, i.e. they are compatible. Since $\mathcal B$ is countable, it follows that $\mathbb P_1$ is $\sigma$-centered.

 For each $\gamma\in \alpha$, we define $D_\gamma=\{\langle s, F\rangle: g_\gamma\in F\}$. Then $D_\gamma$ is dense in $(\mathbb{P}_1, \leq_1)$. For any $\langle s, F\rangle\in (\mathbb{P}_1, \leq_1)$, pick $\langle s, F\cup \{\gamma\}\rangle$ which satisfies that $\langle s, F\cup \{\gamma\}\rangle\leq \langle s, F\rangle$. For each $n\in \omega$, we define $E_n=\{\langle s, F\rangle\in\mathbb{P}_1: \pi_1(s)\nsubseteq n\}$ where $\pi_1(s)=\{i: x_i\in s\}$. Take $ \langle s, F\rangle\in\mathbb{P}_1$ and $i>n$. Since $g_\gamma$ is a sequence converging to $\infty$, $\textrm{ran}(g_\gamma)\cap \{f_\alpha(i,m):m\in\omega\}$ is finite for each $\gamma<\alpha$. Since $F$ is finite, there is an $m_i\in \omega$ such that $f_\alpha(i,m_i)\not\in \textrm{ran}(g_\gamma)$ whenever $i>n$ for all $g_\gamma\in F$. Let $x_i=(i,m_i)$. Then $\langle s\cup \{x_i\}, F\rangle\leq \langle s, F\rangle$. Hence, $E_n$ is dense for each $n\in \omega$.

 By $\mathfrak p=\mathfrak c$, there is a filter $G_1$ in $\mathbb{P}_1$ such that $G$ has nonempty intersection with $D_\gamma$ and $E_n$ for each $\gamma\in \alpha$ and $n\in \omega$. Then we define $d_{G_1}=\bigcup\{s: \exists F (\langle s, F\rangle\in G_1)\}$. First, we claim that if $x, y\in d_{G_1}$, then $\pi_1(x)\neq \pi_1(y)$. Otherwise, if $x\in s$ and $y\in s'$, then $\langle s, F\rangle$ and $\langle s', F'\rangle$ are not compatible where $F$ and $F'$ are the sets such that $\langle s, F\rangle$ and $\langle s', F'\rangle$ are in $G_1$. This leads to a contradiction since $G_1$ is a filter. Secondly, since $G_1\cap E_n\neq\emptyset$, $d_{G_1}$ is a sequence converging to $\infty$. So condition (ii) is satisfied. Lastly, fix any $\langle s, F\rangle\in G_1$. We claim that $\{f_{\alpha}(x): x\in d_{G_1}\}\cap \{g_\gamma(n):n\in\omega\}\subset \{f_\alpha(x_i): x_i\in s\}$. Take any $\langle s', F'\rangle\in G$. By the compatibility of $\langle s, F\rangle$ and $\langle s', F'\rangle$, $\{f_{\alpha}(x): x\in s'\}\cap \{g_\gamma(n):n\in\omega\}\subset \{f_\alpha(x_i): x_i\in s\}$.  Since $G_1\cap D_\gamma\neq \emptyset$, $\{f_{\alpha}(x): x\in d_{G_1}\}\cap \{g_\gamma(n):n\in\omega\}$ is finite for any $\gamma<\alpha$. Let $U_\alpha=\{f_\alpha(x_i): x_i \in d_{G_1}\}$.  Then condition iii) is satisfied.

\medskip

 Now consider $\varphi_\alpha$; if it is not l.s.c.  when $X$ has the topology $\tau_{\alpha+1}$ generated by $\tau_\alpha\cup\{U_\alpha\}$,  we let $g_\alpha$ be the sequence which is constant $\infty$.  Otherwise, we define a continuous selection as follows.
 Let $\{O_\gamma: \gamma<\kappa\}$ be a local base at $\infty$ in $(X,\tau_{\alpha+1})$ where $\kappa=|\alpha|$ which is clearly $<\mathfrak c$. For each $\gamma<\kappa$, let $L_\gamma=X\setminus U_\gamma$. We define $\mathcal{H}$ to be the collection of finite subsets of  $\{L_\gamma: \gamma<\kappa\}$. Let $\mathcal{B}_2$ be the collection of finite subsets of $\omega\times X\setminus \{\infty\}$ such that for any $s\in \mathcal{B}_2$, $\pi_1(x)\neq\pi_1(y)$ if $x, y\in S$. Define $\mathbb{P}_2=\{\langle s, H\rangle: H\in \mathcal{H}\text{ and } s\in \mathcal{B}_2\}$. Then we define  an order $\leq_2$ on $\mathbb{P}_2$ as follows:
   $$\langle s, H\rangle\leq_2 \langle s', H'\rangle \leftrightarrow s\supset s', H\supset H', \forall L\in H'(L\cap \{\pi_2(x): x\in s\}\subseteq \{\pi_2(x): x\in s'\}). $$

  For each $\gamma\in \kappa$, let $D_\gamma=\{\langle s, H\rangle: L_\gamma\in H\}$. For each $n\in \omega$, let $E_n=\{\langle s, H\rangle\in\mathbb{P}_2: \pi_1(s)\supset n\}$ where $\pi_1(s)=\{i: i=\pi_(x) \text{ for some } x\in s\}$. It is straightforward as above to verify that $D_\gamma$ and $E_n$ are dense in $\mathbb{P}_2$ for each $\gamma<\kappa$, $n\in \omega$ and that $\mathbb{P}_2$ is $\sigma$-centered.  Again by $\mathfrak p=\mathfrak c$, there is a filter $G_2$ in $\mathbb{P}_2$ such that $G_2$ has nonempty intersection with $D_\gamma$ and $E_n$ for each $\gamma\in \kappa$ and $n\in \omega$. Then we define $d_{G_2}=\bigcup\{s: \exists H\in\mathcal{F} \langle s, H\rangle\in G)\}$. Since $G\cap E_n\neq \emptyset$ for each $n$, the domain of $d_{G_2}$ is $\omega$. Since $G_2\cap D_\gamma\neq \emptyset$ for each $\gamma<\kappa$, $\pi_2(d_G)\cap F_\gamma$ is finite for each $\gamma<\kappa$, i.e. $U_\gamma\setminus \pi_2(d_{G_2})$ is finite. Define $g_\alpha=d_G\cup \{(\omega, \infty)\}$ which is clearly a sequence converging to $\infty$.

  \medskip


We define $\tau_\alpha$, $U_\alpha$, and $g_\alpha$ for all $\alpha<\mathfrak c$ above.
Let $\tau$ be the topology generated by $\tau_0\cup \{U_\alpha:\alpha<\mathfrak c\}$; we show that $(X,\tau)$ is strongly $L$-selective but not
$M_\omega$-selective.

The mapping $\psi$ defined at the beginning of the construction is clearly closed; we check that it is l.s.c..   A basic neighborhood $O$ of $\infty$ in $X$ has the form
 $$[X\setminus (k\times \omega^2)]\cap\bigcap_{\alpha\in F}U_\alpha$$  where $F$ is a finite subset of $\mathfrak c$. The set  $X\setminus (k\times \omega^2)$
 contains $L_{mn}$ for every $m\geq k$, while each $U_\alpha$ contains all but at most one point of every $L_{mn}$.  It follows that $\psi^{-1}(O)=M_\omega\setminus (k\times \omega)$ which of course is open in $M_\omega$.  Thus $\psi$ is l.s.c..

Suppose $f:M_\omega\rightarrow X$ is a selection for $\psi$.  Then $f=f_\alpha$ for some $\alpha$.  But by condition (ii) above, there is a sequence $x_i$, $i\in\omega$, of points of $M_\omega$ converging to $\infty$ but $f_\alpha(x_i)$ does not converge in $X$.  So $f$ is not continuous.  Thus there is no continuous selection for $\psi$ and $X$ is therefore not $M_\omega$-selective.

It remains to check that $X$ is strongly $L$-selective.  Suppose $\varphi$ is an l.s.c. function from $\omega+1$ to the nonempty subsets of $(X,\tau)$.  If $\varphi(\omega)$ contains some isolated point $x$, it follows from l.s.c. that $x\in \varphi(n)$ for almost all $n\in\omega$, so there is a continuous selection which is almost constant.
Hence we may suppose that $\varphi(\omega)=\{\infty\}$. Then $\varphi=\varphi_\alpha$ for some $\alpha<\mathfrak c$ which is l.s.c. in the weaker topology $\tau_\alpha$. At stage $\alpha$, we defined a continuous selection $g_\alpha$ which is continuous at later stages. Then $g_\alpha$ is a continuous selection for $\varphi$ when $X$ has topology $\tau$.  Thus $(X,\tau)$ is strongly $L$-selective.   \end{proof}


\section{$L$-selective vs. strongly $L$-selective}\label{L-not-strong-L}

By Proposition~\ref{prod_go_ctbl}, the ordinal space $\omega_1+1$ is $L$-selective (even $C$-selective), but it is not strongly $L$-selective since it is not Fr\'echet.
In this section we will show that the statement ``A space is strongly $L$-selective iff it is $L$-selective and Fr\'echet" is consistent with and independent of ZFC.  The following is a key proposition.

\begin{prop} Suppose $X$ is a Fr\'echet $L$-selective space, and that every countable subspace of $X$ is first countable.
Then $X$ is strongly $L$-selective. \end{prop}

\begin{proof}
Suppose $X$ satisfies the hypotheses, and let $\varphi:\omega+1\to \mathcal P(X)\sm\{\0\}$ be l.s.c. Also assume $\varphi(\omega)$ is a singleton $\{p\}$ (which we may, since the modified function is l.s.c. too, for any choice of $p\in \varphi(\omega)$).  Define
$\cl{\varphi}:\omega+1\to \mathcal F(X)$ by $\cl{\varphi}(k)=\cl{\varphi(k)}$ for each $k\in\omega+1$.  If $U$ is open in $X$, then for each $k\in \omega+1$
we have $\varphi(k)\cap U\neq\0$ iff $\cl{\varphi(k)}\cap U\neq \0$, and so $\cl{\varphi}^{-1}(U)=\varphi^{-1}(U)$ which is open.
Thus $\cl{\varphi}$ is l.s.c.  Since $X$ is $L$-selective, there is a continuous selection $f:\omega+1\to X$.

Since $X$ is Fr\'echet and $f(n)\in \cl{\varphi(n)}$, we may for each $n\in\omega$ choose a sequence $\{p_{nm}:m\in\omega\}$ converging to $f(n)$ with $p_{nm}\in \varphi(n)$ for each $m\in \omega$.  Now let $X'=\{p_{nm}:n,m\in\omega\}\cup\{p\}$, and define $\varphi':\omega+1\to \mathcal P(X')\sm\{\0\}$ by $\varphi'(\omega)=\{p\}$ and $\varphi'(n)=\{p_{nm}:m\in\omega\}$ for
each $n\in\omega$.  Since every neighborhood of $p$ in $X$ contains $f(n)$, and hence meets $\varphi'(n)$, for all sufficiently large $n$, it is easy to check that $\varphi'$ is l.s.c.  Let $g:\omega+1\to X'$ be a continuous selection
for $\varphi'$; note that $g$ exists since $X'$ is first countable.  Since $\varphi'(n)\subset \varphi(n)$ for each $n$, it follows that $g$ is a continuous selection for $\varphi$ as well.
\end{proof}

\begin{cor}  It is consistent with ZFC that the following two statements are equivalent for any space $X$:
\begin{enumerate}
  \item[(1)] $X$ is strongly $L$-selective;
  \item[(2)] $X$ is $L$-selective and Fr\'echet.
\end{enumerate}
\end{cor}

\begin{proof}
  We already know that (1) imples (2) in any model.  Now consider a model in which every countable Fr\'echet $\alpha_1$-space is first countable (e.g., the model of
  \cite{ds}), and suppose in this model that $X$ is $L$-selective and Fr\'echet.  By Corollary~\ref{count_subspace} and Proposition~\ref{alpha_1}, every countable subspace of $X$
  is Fr\'echet and $\alpha_1$, hence first countable.  Then by the previous proposition, $X$ is strongly $L$-selective.  Hence (1) and (2) are equivalent in this model.
\end{proof}

  Now we show that the equivalence of (1) and (2) above is false in any model of CH.
 Assuming CH, we will construct a countable $L$-selective (hence Fr\'echet) space $X$ which is not strongly $L$-selective.
Our construction will be aided by considering a selective ultrafilter on $\omega$.  Selective ultrafilters are sometimes called Ramsey ultrafilters because of the following characterization due to Kunen: $\mathcal F$ is selective iff for every 2-coloring of $[\omega]^2$ there is a homogenous $F\in\mathcal F$.\footnote{For a set $A$, $[A]^2$ is the set of all 2-element subsets of $A$.  A 2-coloring of a set is a partition of the set into two pieces. A set $F$ is homogeneous for a 2-coloring of $[A]^2$ if the elements of $[F]^2$ all have the same color.} In fact this is the property of selective ultrafilters which we use below. Walter Rudin proved that selective ultrafilters exist under CH; they also exist under various other assumptions, for example Martin's Axiom.  Selective ultrafilters are $P$-points, a fact we also use in our arguments.

\begin{ex}\label{l-not-strong-l}\upshape{(CH)} There is a countable $L$-selective space which is not strongly $L$-selective. \end{ex}

\begin{proof} The set for our space $X$ is $((\omega+1)\times\omega)\cup\{\infty\}$.  Let $\tau_0$ be the topology on $X$ such that points of $\omega\times\omega$ are isolated, a local base at $(\omega,n)$ is $B(\omega,n,k)=\{(m,n):m\geq k\}$, and a local base at $\infty$ is $\{X\setminus ((\omega+1)\times n ): n\in\omega\}$. That is, $(X,\tau_0)$ is the one-point compactification of the topological sum of countably many convergent sequences.

The plan is to modify the topology of $(X,\tau_0)$ by adding open sets to $\tau_0$ in an induction of length $\omega_1$ to make it $(\omega+1)$-selective but not strongly $(\omega+1)$-selective. Define $\psi:\omega+1\rightarrow {\mathcal P}(X)\setminus \{\emptyset\}$ by $\psi(\omega)=\{\infty\}$ and $\psi(n)=L_n$, where $L_n=\omega\times \{n\}$.  Note that $L_n$ is not closed in $X$; it has one  limit point, namely $(\omega,n)$.  When finished with our induction, we want this $\psi$ to be l.s.c. (which is equivalent to every neighborhood of $\infty$ meeting all but finitely-many $L_n$'s) but have no continuous selection.  Obviously, for the space to be $(\omega+1)$-selective, we will need that
the sequence $\{(\omega,n):n\in\omega\}$ still converges to $\infty$ when we are done.

Let $\{f_\alpha:\alpha<\omega_1\}$ list all functions $f:\omega+1\rightarrow X$ such that $f(\omega)=\infty$ and $f(n)\in L_n$ for each $n\in\omega$.  We also let $\{\varphi_\alpha: \alpha<\omega_1\}$ list all functions $\varphi:\omega+1\rightarrow {\mathcal P}(X)\setminus \{\emptyset\}$ such that $\varphi(\omega)=\{\infty\}$ and each specific function appears uncountably often.  Since $X$ is countable, $\mathcal P(X)$ and $\mathcal{P}(X)^\omega$ have cardinality continuum $\mathfrak c=\omega_1$, so such a listing is possible.
At stage $\alpha$ of the construction, we
will  make $f_\alpha$ not continuous by declaring a certain subset of its range closed discrete.     We then look at $\varphi_\alpha$.  If $\varphi_\alpha$ is not closed and l.s.c. in the topology as defined so far, we do nothing.  If it is, we do one of two things: add an open set to $X$ so that $\varphi_\alpha$ is no longer l.s.c., or specify a continuous selection $g_\alpha$ with the intention of keeping it continuous through all later stages.  We will use a selective ultrafilter $\mathcal F$ to make sure the closed discrete sets we add don't interfere with the continuous selections.  The key idea is to make the range of $g_\alpha$ meet only non-ultrafilter many $L_n$'s, i.e., $A_\a=\{k\in\o: (\textrm{ran } g_\a)\cap L_k\neq \0\}\not\in\mathcal F$. Selectivity of $\mathcal F$ is helpful to get this.  Then we can use that $\mathcal F$ is a $P$-point to make $f_\a$ discontinuous by declaring $\{f_a(i):i\in F_\a\}$ closed, where $F_\a\in \mathcal F$ meets each $A_\b$, $\b<\a$, in a finite set.  This preserves continuity of the previously defined $g_\b$'s.

We now give the details.   Fix a selective ultrafilter $\mathcal F$. Suppose $\alpha<\omega_1$ and for each $\beta<\alpha$ we have constructed $U_\beta\subset X$, $V_\beta\subset X$, $g_\beta:\omega+1\rightarrow X$, $, F_\b\in \mathcal F$, $A_\b\in \mathcal P(\omega)\sm \mathcal F$, $E_\b\in \mathcal P(\omega)$, and topology $\tau_\beta$ on $X$ satisfying:

 \begin{enumerate}
   \item[(i)] $\tau_\beta$ is the topology generated by $\tau_0\cup\{U_\gamma:\gamma<\beta\}\cup\{V_\gamma:\gamma<\beta\}$;
   \item[(ii)] $F_\beta\in \mathcal F$ and $U_\beta=X\setminus \{f_\beta(i):i\in F_\beta\}$;
   \item[(iii)] If $\varphi_\beta$ is not closed l.s.c. when $X$ has the topology generated by $\tau_\beta\cup \{U_\beta\}$, then $g_\beta$ is constant $\infty$ and $V_\beta=X$;
   \item[(iv)] If $\varphi_\beta$ is closed l.s.c. when $X$ has the topology generated by $\tau_\beta\cup \{U_\beta\}$, then $\varphi^{-1}(V_\beta)$ is not open and $g_\beta$ is constant $\infty$, or $g_\beta$ is a continuous selection such that $A_\beta=\{k\in\o: (\textrm{ran } g_\beta)\cap L_k\neq \0\}$ is not in $\mathcal F$ .
  \item[(v)] If $V_\b$ is defined as in (iv), then $V_\b=X\sm \bigcup_{n\in E_\b}K_n$, where $K_n$ is a finite subset of $L_n$;
      \item[(vi)] If $\gamma<\beta$, then $E_\beta\cap A_\g$ and $F_\beta\cap A_\g$ are finite.
  \end{enumerate}
 We check that (i)--(vi) can be had at step $\alpha$.  Note that $\tau_\alpha$ is given by (i). Since $\mathcal F$ is a
 $P$-point, we may choose $A\not\in \mathcal F$ such that $A_\b\subset^* A$  for each $\b<\a$.  Then let $F_\a=\omega\sm A$ and  $U_\alpha=X\setminus \{f_\alpha(i):i\in F_\a\}$, and let $\tau_\alpha'$ be the topology generated by $\tau_\alpha\cup\{U_\alpha\}$.  Then $f_\alpha$ is not continuous when $X$ is given topology $\tau_\a'$ or any finer topology. Also, since $F_\a\cap A_\b$ is finite when $\b<\a$, it follows from (iv) that $\{f_\a(i):i\in F_\a\}$ meets the range of $g_\b$ in at most a finite set, so $g_\b$ remains continuous in $\tau'$.

 Now consider $\varphi_\alpha$.  If $\varphi_\alpha$ is not closed and l.s.c., let $g_\beta$ be constant $\infty$ and $V_\beta=X$ as required by (iii).  Suppose on the other hand that $\varphi_\alpha$ is closed l.s.c..  We consider two cases.

 {\it Case 1. For some clopen (in $\tau_\alpha'$) set $U$ containing $\infty$, the set $E(U)=\{k\in \omega\setminus A: \exists n (\varphi_\a(n)\cap U\subset L_k)\}$ is infinite.}\footnote{It is worth noting that $\varphi_\a(n)\cap U$ is closed, so if it is contained in $L_k$ then it is finite.}
 In this case, we proceed to add an open set making $\varphi_\a$ not l.s.c. as follows.  Let $U$ be such that $E(U)$ is infinite, and let $E_\a=E(U)$. Since $E_\a\cap A=\0$, we have that $E_\a\cap A_\b$ is finite when $\b<\a$.  For each $k\in E_\a$, choose $n_k\in \omega$ such that $\varphi_\a(n_k)\cap U\subset L_k$,   and let $V_\alpha=X\sm \bigcup_{k\in E_\a}\varphi_\a(n_k)\cap U$.  Then for each $k\in E_\a$,  $\varphi_\alpha(n_k)\cap U\cap V_\alpha=\emptyset$.  Hence $\varphi_\alpha^{-1}(U\cap V_\alpha)$ is not open in $\omega +1$, so $\varphi_\alpha$ is not l.s.c. when $X$ is given the topology $\tau_{\alpha+1}$ generated by $\tau_\alpha'\cup \{V_\alpha\}$.  Since $V_\alpha$ contains $(\omega+1)\times A$, it also contains a tail of every sequence $g_\b$, $\b<\a$, so $g_\b$ remains continuous.
Let $g_\a$ be constant $\infty$.

{\it Case 2. The set $E(U)$ is finite for every $\tau_\alpha'$-open neighborhood $U$ of $\infty$.}  We then  define a continuous selection $g_\alpha$ as follows.    Since $\tau_\a'$ has a countable base and is 0-dimensional, we can let $O_0, O_1,...$ be a decreasing clopen neighborhood base at $\infty$.  Note that for each fixed $k$ and $n$, the set $\{i\in\omega\sm A: \varphi_\a(i)\cap O_n\subset L_k\}$ is finite, for otherwise $\varphi_\a^{-1}(O_n\cap (X\sm \cl{L}_k))$ would not be open. Since $E(O_n)$ is also finite, we can find natural numbers $k_0<k_1<...$ such that $$i\geq k_n \imp \not\exists k\in\omega\setminus A(\varphi_a(i)\cap O_n\subset L_k).$$  Hence for $i\geq k_n$,  $\varphi_\a(i)\cap O_n$ satisfies at least one of the following:
\begin{enumerate}
  \item[(i)] $ \varphi_a(i)\cap O_n$ meets $((\omega+1)\times A)\cup\{\infty\}$;
  \item[(ii)] $(\omega,k)\in \varphi_\a(i)\cap O_n$ for some $k\in\omega$;
  \item[(iii)] $\varphi_\a(i)\cap O_n$ meets at least two $L_k$'s, $k\not\in A$.
\end{enumerate}
If (iii) occurs, and if also $i<k_{n+1}$, let $d_i\in [\omega\setminus A]^2$ such that $\varphi_0(i)\cap O_n\cap L_k\neq \emptyset$ for each $k\in d_i$.  Then we use selectivity of $\mathcal F$ to show there is a set $D\not\in \mathcal F$ which meets every $d_i$ that has been defined.  To see this, color each $d_i$ black and every other element of $[\omega]^2$ white, and let $C\in \mathcal F$ be homogeneous for this coloring.  Note that for each $k$, $O_n\cap \cl{L}_k=\0$ for all sufficiently large $n$.  Thus no $k$ is in infinitely many $d_i$. It follows that $C$ is homogeneous for white, and so $D=\omega\sm C$ meets each $d_i$.    Now for $k_n\leq i<k_{n+1}$, make the following selection of $x_i$ in $\varphi_\a(i)\cap O_n$: $x_i\in [((\omega+1)\times A)\cup\{\infty\}]$ if (i) occurs,  $x_i=(\omega,k)$ for some $k$ if (i) fails but (ii) holds, and $x_i\in L_k$ for some $k\in d_i\cap D$ if neither (i) nor (ii) hold.  For $i<k_0$, choose any $x_i\in \varphi_\a(i)$.   Then $g_\alpha(i)=x_i$ and $g_\alpha(\omega)=\infty$ defines a continuous selection.  Let $A_\a=\{k\in\o: (\textrm{ran } g_\a)\cap L_k\neq \0\}$; then $A_\a\subseteq A\cup D$, so is not in $\mathcal F$ .

We have now defined $\tau_\alpha$ for all $\alpha<\omega_1$.  Let $\tau$ be the topology generated by the subbase $\tau_0\cup \{U_\a:\a<\omega_1\}\cup\{V_\a:\a<\omega_1\}$; we show that $(X,\tau)$ is $(\omega+1)$-selective but not strongly so. Note that each $U_\a$ and $V_\a$ contains all of $X$ except for finitely many points of some $L_k$'s. Thus the topology of each $\cl{L}_k$ is the same in $\tau$ as it was in $\tau_0$, and every neighborhood of $\infty$ almost contains $L_k$ for all sufficiently large $k$.  It follows that $\psi$ is l.s.c. when $X$ has topology $\tau$.  There is no continuous selection for $\psi$ since each potential one was made discontinuous at some stage of the induction.  So $(X,\tau)$ is not strongly $(\omega+1)$-selective.

It remains to check that $X$ is $(\omega+1)$-selective.  Suppose $\varphi$ is an l.s.c. function from $\omega+1$ to the nonempty closed subsets of $(X,\tau)$.  If $\varphi(\omega)$ meets $\cl{L}_k$ for some $k$, it follows from l.s.c. that $\varphi(n)$ meets $\cl{L}_k$ for all sufficiently large $n$.  Since $\cl{L}_k$ is $C$-selective (it is first countable), it easily follows that $\varphi$ has a continuous selection.

Hence we may suppose that $\varphi(\omega)=\{\infty\}$.  Note that every basic open subset of $(X,\tau)$ appears in some $\tau_\alpha$ with $\alpha<\omega_1$.  Then since $X$ is countable, it follows that every open set (and so every closed set too) appears at some stage before $\omega_1$.  Since each
function  from $\omega+1$ to ${\mathcal P}(X)\setminus \emptyset$ with $\varphi(\omega)=\{\infty\}$ appears uncountably often in the listing $\{\varphi_\alpha: \alpha<\omega_1\}$, there is $\alpha<\omega_1$ such that $\varphi_\alpha= \varphi$ and $\varphi(n)$ is closed in $(X,\tau_\alpha)$ for all $n\in\omega$.  Since $\varphi$ is l.s.c. when $X$
has topology $\tau$, it is also l.s.c. when $X$ has the weaker topology $\tau_\alpha$.  So at stage $\alpha$, $\varphi_\alpha$ is closed l.s.c.,  thus at this stage we either added an open set which made $\varphi_\alpha$ not l.s.c.,
or we defined a continuous selection $g_\alpha$ which we made sure remained continuous at later stages.  Since $\varphi_\alpha=\varphi$ is still l.s.c. when $X$ has topology
$\tau_{\alpha+1}$, we must have done the latter, not the former.  Then $g_\alpha$ is a continuous selection for $\varphi$ when $X$ has topology $\tau$.  Thus $(X,\tau)$ is
$(\omega+1)$-selective.  \end{proof}

\section{Summary and Open Questions}\label{open_q}

Many of the results of this paper are summarized in the following chart. The solid arrow means implication in ZFC, the dashed arrow means implication consistently, and the dashed arrow with $\times$ means that there is a counterexample consistently. Also FU stands for Fr\'echet-Urysohn, and recall that $CFC$ spaces are those in which countable subspaces are first countable.

\medskip
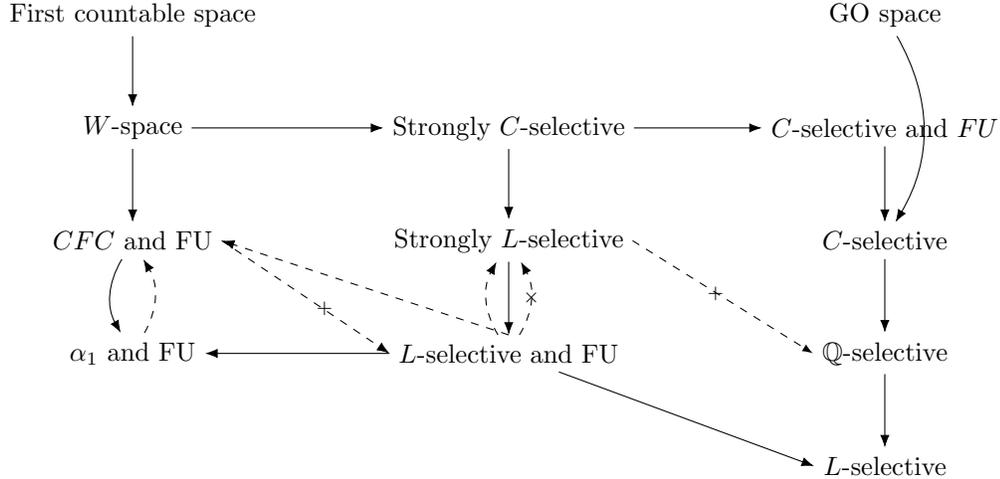
\begin{tikzpicture}[auto,prop/.style={rectangle}]

\node (P) at (0, 7) {First countable space};
\node (G) at (10, 7) {GO space};

\node (P00) at (0, 5.5) {$W$-space};

\node (P1) at (5, 5.5) {Strongly $C$-selective};
\node (P02) at (10, 5.5) {$C$-selective and $FU$};

\node (P10) at (0, 4) {$CFC$ and FU};
\node (P11) at (5, 4) {Strongly $L$-selective};
\node (P12) at (10, 4) {$C$-selective};

\node (P20) at (0, 2.5) {$\alpha_1$ and FU};
\node (P21) at (5, 2.5) {$L$-selective and FU};
\node (P22) at (10, 2.5) {$\mathbb{Q}$-selective};

\node (P32) at (10, 1) {$L$-selective};


\draw[-Latex] (P)--(P00);
\draw[-Latex, bend left] (G) edge (P12);
\draw[-Latex] (P00)--(P10);

\draw[-Latex, bend right] (P10) edge (P20);

\draw[dashed, -Latex, bend right] (P20) edge (P10);
\draw[-Latex] (P21)--(P32.west);
\draw[-Latex] (P21)--(P20);
\draw[-Latex] (P22) edge (P32);
\draw[-Latex] (P12)--(P22);

\draw[-Latex] (P02)--(P12);

\draw[-Latex] (P1)--(P02);
\draw[-Latex] (P1)--(P11);
\draw[-Latex] (P00)--(P1);

\draw[dashed, -latex] (P11.east)--(P22.west);
\node (not2) at (7.75,3.30)  {\scriptsize +};

\draw[dashed, -Latex, bend left] (P21) edge (P11);

\draw[-Latex] (P11) edge (P21);
\draw[dashed, -Latex, bend right] (P21) edge (P11);

\node at (5.3,3.25)  {\scriptsize $\times$};

\draw[dashed, -Latex] (P21.north) -- (P10.east);

\draw[dashed, -Latex] (P10.east) -- (P21.west);
\node (not3) at (2.55,3.1)  {\scriptsize +};


\end{tikzpicture}

We list some open questions below.

\begin{qu} Is there a Fr\'echet $C$-selective space which is not strongly $C$-selective? \end{qu}

\begin{qu}Is there a $\mathbb Q$-selective space which is not $C$-selective?\end{qu}

\begin{qu}
Is there in ZFC a (strongly) $L$-selective space which is not $\mathbb Q$-selective?  Not $C$-selective?
\end{qu}

\begin{qu}
Is there in ZFC a $CFC$ Fr\'echet space which is not $L$-selective? Not $C$-selective?
\end{qu}


\begin{thebibliography}{99}


\bibitem{aa}  A.V. Arhangel'ski\v{i}, ``Frequency spectrum of a topological space and classification", Doklady Akademii Nauk SSSR \textbf{206} (1972) 265--268.

\bibitem{ds} A. Dow and J. Steprans, ``Countable Fréchet $\alpha_1$-spaces may be first countable", Arch. Math. Logic \textbf{32}(1992), no. 1, 33–50.

\bibitem{gn} J. Gerlits and Zs. Nagy, ``Some properties of C(X), I", Topology Appl. 14 (1982) 151 – 161.

\bibitem{gg}  G. Gruenhage, ``Infinite games and generalizations of first countable spaces", General Topology and its Applications \textbf{6} (1976) 339--351.

\bibitem{m81} E. Michael, `` Continuous selections and countable sets", Fundamenta Mathematicae, \textbf{111} (1981) 1--10.

\bibitem{RS94} D. Repovs, P.V. Semenov, ``Michael's theory of continuous selection. Development and applications", Russian Math. Surveys, \textbf{49} (1994) 157--196.

\bibitem{sj}  R. Shen, H. J. K. Junnila, ``On s-reflexive spaces and continuous selections", Chinese Annals of Mathematics, \textbf{36B} (2015) 181--194.


\end{thebibliography}
 \end{document}